\documentclass[runningheads]{llncs}
\usepackage[utf8]{inputenc}
\usepackage{amsfonts,amsmath,amssymb,graphicx,color,algorithm}
\usepackage[colorlinks=true]{hyperref}
\usepackage[noend]{algpseudocode}
\usepackage[left=1in,right=1in,top=1in,bottom=1in]{geometry}
\usepackage{cleveref}
\usepackage{accents}

\usepackage{booktabs}

\usepackage[belowskip=-15pt,aboveskip=0pt]{caption}
\usepackage{float}
\usepackage{capt-of}

\newcommand{\Reals}{\mathbb{R}}
\newcommand{\cB}{\mathcal{B}}
\newcommand{\Var}{\mathbb{V}}
\newcommand{\Nats}{\mathbb{N}_{+}}
\newcommand{\Exp}{\mathbb{E}}

\title{Estimating Computational Noise on Parametric Curves}
\author{Matt Menickelly}
\institute{Argonne National Laboratory}
\date{\today}

\begin{document}

\maketitle

\section{Introduction}
We consider \texttt{ECNoise} \cite{more2011estimating}, a practical tool for estimating the magnitude of noise in evaluations of a black-box function. 
Recent developments in numerical optimization algorithms have seen increased usage of \texttt{ECNoise} as a subroutine to provide a solver with noise level estimates, so that the solver might somehow proportionally adjust for noise. 
Particularly motivated by problems in computationally expensive derivative-free optimization,
we question a fundamental assumption made in the original development of \texttt{ECNoise}, particularly the assumption that the set of points provided to \texttt{ECNoise} must satisfy fairly restrictive geometric conditions (in particular, that the points be collinear and equally spaced). 
Driven by prior practical experience, we show that in many situations, noise estimates obtained from providing an arbitrary (that is, not collinear) geometry of points as input to \texttt{ECNoise} are often indistinguishable from noise estimates obtained from using the standard (collinear and equally spaced) geometry. 
We analyze this via parametric curves that interpolate the arbitrary input points (\Cref{sec:theory}). 
The analysis provides insight into the circumstances in which one can expect arbitrary point selection to cause significant degradation of \texttt{ECNoise} (\Cref{sec:understanding}). 
Moreover, the analysis suggests a practical means (the solution of a small mixed integer linear program) by which one can gradually adjust an initial arbitrary point selection to yield better noise estimates with higher probability (\Cref{sec:milp}). 

We begin by properly introducing \texttt{ECNoise} in \Cref{sec:ecnoise}, and then elaborating on our motivation in \Cref{sec:motivation}. 

\section{Introduction to \texttt{ECNoise}} \label{sec:ecnoise}
\texttt{ECNoise} \cite{more2011estimating} is a practical tool for estimating the magnitude of noise in evaluations of a (black-box) function $f:\Reals^n\to\Reals$, given reasonable assumptions on an underlying noise model. 
To make the noise model immediately concrete, 
we assume that there exists some $m$-times continuously differentiable ``ground truth" function $f_s:\Reals^n\to\Reals$. 
Observations of $f_s(x)$ can never be directly made given a query $x$;
instead, we only observe
\begin{equation}
\label{eq:noisemodel}
f(x) = f_s(x) + \xi(x),
\end{equation}
where $\xi(x)$ is a noise term. 
While $\xi$ may be a function of $x$, we make the simplifying assumption for the purpose of analysis that in a small neighborhood around a base point $y^0\in\Reals^n$, denoted $\cB(y^0,\delta) = \{x: \|x -y^0\| \leq \delta\}$,
each evaluation of $\xi(x)$ for $x\in\cB(y^0,\delta)$ is an independent, identically distributed realization of a random variable $\Xi(y^0,\delta)$. 
Given the noise model in \eqref{eq:noisemodel}, we define the \emph{noise level} of $f$ over the neighborhood $\cB(y^0,\delta)$
\begin{equation}
    \label{eq:noise_level}
    \epsilon_f := \sqrt{\Var\left[\Xi(y^0,\delta)\right]},
\end{equation}
where $\Var\left[\cdot\right]$ denotes the variance of a random variable. 

While the \texttt{MATLAB} software implementation of \texttt{ECNoise} involves several additional, and practical, heuristic extensions, the core idea behind \texttt{ECNoise} is conceptually simple. 
In essence, \texttt{ECNoise} follows a technique proposed by Hamming (see, e.g., \cite{hamming2012introduction}[Chapter 6]).
Given a \emph{base point} $y^0$, a number of points $m$, a \emph{differencing interval} $h$, and some unit direction $d\in\Reals^n$, an evaluation of $f$ is obtained at each point in the set $\{y^0 + jhd: j = 0, 1, \dots, m - 1\}$.
A \emph{differencing table} is then constructed using these function values.
Concretely, letting $y^j = y^0 + jhd$ for $j = 0,1,\dots,m$,
we define a $k$th-order difference of $f(y^j)$ via a base case
$\Delta^0 f(y^j) = f(y^j)$, and we then recursively let
\begin{equation}\label{eq:kth_order_diff}
\Delta^{k+1} f(y^j) := \Delta[\Delta^k f(y^j)] := \Delta^k f(y^{j+1}) - \Delta^k f(y^j)
\end{equation}
for all $j \leq m - 1 - k$.
The columns of a differencing table---see \Cref{table:differencing_example} for a concrete example---resulting from this recursion are composed of the $k$th-order differences $\{\Delta^k f(y^j): j = 0, 1, \dots, m-1-k\}$. 
Viewing each column of the table as $m-k$ samples of a $k$th-order difference,  one can estimate
the noise $\epsilon_f$, for any $k\geq 1$, as
\begin{equation}
    \label{eq:eps_f_approx}
    \epsilon_f\approx \displaystyle\sqrt{ \frac{\gamma_k}{m-k} \sum_{j=0}^{m-1-k} (\Delta^k f(y^j))^2}, \quad \gamma_k := (k!)^2/(2k)!.
\end{equation}
At first glance, the reasoning behind why \Cref{eq:eps_f_approx} serves as an approximation to \Cref{eq:noise_level} is opaque.
Nonetheless, we will recover this reasoning as a special case of analysis performed in this paper; see \Cref{thm:noise_estimate} and \Cref{thm:kcd_equallyspaced}.

    \begin{table}[h!]
    \centering
    \caption{\label{table:differencing_example} Example of a differencing table from a single trial of the experiment described in \Cref{sec:motivation}, with multiplicative noise distributed normally with zero mean and standard deviation $10^{-3}$.
    Because this experiment deals with multiplicative noise, we add an additional row to show that the normalized noise estimates are close to the standard deviation, which is reasonable given the definition of $\epsilon_f$ in \eqref{eq:noise_level}.}
\begin{tabular}{r|cccccc}
                            & $f=\Delta^0 f$ & $\Delta^1 f$ & $\Delta^2 f$ & $\Delta^3 f$ & $\Delta^4 f$ & $\Delta^5 f$ \\
                            \hline
$y^0$                       & 328.3654       & 0.9293       & -1.8141      & 2.8776       & -4.4118      & 6.8260       \\
$y^1$                       & 329.2947       & -0.8848      & 1.0635       & -1.5343      & 2.4141       &              \\
$y^2$                       & 328.4099       & 0.1787       & -0.4708      & 0.8799       &              &              \\
$y^3$                       & 328.5886       & -0.2921      & 0.4091       &              &              &              \\
$y^4$                       & 328.2965       & 0.1169       &              &              &              &              \\
$y^5$                       & 328.4134       &              &              &              &              &              \\ \hline
Value of \eqref{eq:eps_f_approx}:             &                & 0.4216       & 0.4477       & 0.4361       & 0.4250       & 0.4300       \\ \hline
(Value of \eqref{eq:eps_f_approx}) / $f(y^0)$ &                & 0.0013       & 0.0014       & 0.0013       & 0.0013       & 0.0013      
\end{tabular}
\end{table}

\section{Motivation}\label{sec:motivation}
A recent trend in some algorithms for numerical optimization has seen attempts to locally estimate the magnitude of noise using \texttt{ECNoise}. 
Such an application of \texttt{ECNoise} was originally suggested by its developers in \cite{more2012estimating}, where an optimal finite differencing interval in the presence of noise was suggested as $h^* = 8^{1/4}\sqrt{(\epsilon_f/\mu)}$, where $\mu$ is a coarse approximation to a second directional derivative in a given direction and $\epsilon_f$ is the noise level  \eqref{eq:noise_level}. 
Possible extensions and improvements to methods for determining optimal finite differencing intervals were suggested and tested in \cite{shi2022adaptive}.

Although \cite{more2012estimating} was concerned primarily with error bounds for finite differencing approximations of directional derivatives of noisy functions, 
the authors were transparent in their development that their estimation procedures could be incorporated in optimization algorithms. 
More recently, 
\cite{berahas2019derivative} considered a quasi-Newton approach to derivative-free optimization that incorporates \texttt{ECNoise} and logic similar to that in \cite{more2012estimating} to determine reasonable finite differencing parameters. 
The quasi-Newton method of \cite{berahas2019derivative} re-estimates the noise level via a new call to \texttt{ECNoise} any time a line search fails. 
Extensive numerical tests for these methods were performed in \cite{shi2023numerical}. 
In derivative-based optimization, a noise-tolerant (L)BFGS method was developed in \cite{shi2022noise}, which suggested using \texttt{ECNoise} to provide an estimate of function value noise as an algorithmic input. 

Our primary motivation in writing this manuscript comes from recent work in model-based derivative-free trust-region methods in the presence of noise \cite{larson2024novel}. 
In this setting, we iteratively construct polynomial interpolation models of a noisy function. 
Using results resembling those for linear interpolation in \cite{berahas2022theoretical} and a characterization of noise employed in a noisy trust-region method analyzed in \cite{cao2023first} that depends on a parameter functionally similar to $\epsilon_f$ in \eqref{eq:noise_level}, we develop a convergent method that fundamentally depends, in each iteration, on providing an estimate to $\epsilon_f$. 
Such an estimation is clearly a role that \texttt{ECNoise} can play within this algorithm. 
However, our forthcoming work was motivated by the noisy optimization problems that arise in variational quantum algorithms, which is within the realm of what one would consider computationally expensive. 
Thus, requiring $m$ function evaluations per run of \texttt{ECNoise} (noting that $m$ is typically chosen between 6 and 12 in practice) may become the dominant cost of an algorithm that specifies budgets in terms of the number of function evaluations, as opposed to the number of iterations. 
Thus, a natural and reasonable question in this setting is 
``Can one reuse noisy function values obtained from an \emph{arbitrary} (that is, not collinear and equally spaced) set of previously evaluated points as input to \texttt{ECNoise}?"
In our setting, the arbitrary set of points comes naturally from the derivative-free trust-region framework (see, e.g., \cite{conn2009introduction}, \cite{larson2019derivative}[Section 2.2]), which stores and maintains a set of interpolation nodes and corresponding (noisy) function evaluations for the purpose of model construction. 

In the remainder of this manuscript we demonstrate that the use of such arbitrary point sets is both practically and theoretically justified.
We provide an initial motivating example. 
We will use the  same stochastic test function initially considered in \cite{more2011estimating}, namely,
$f(x) = (x^\top x)(1 + \xi)$,
that is, a quadratic function perturbed by multiplicative noise. 
This simple test function is of the form \eqref{eq:noisemodel} where the ground truth is given as $f_s(x) = x^\top x$ and the nonconstant noise function is given as $\xi(x)$, which notably vanishes at $x=0$. 
As in \cite{more2011estimating}, we will consider
and $\xi\sim 10^{-3}\mathcal{N}(0,1)$,
where $\mathcal{N}(\mu,\sigma^2)$ denotes a normal distribution with mean $\mu$ and variance $\sigma^2$. 

We perform a simple experiment like the one in \cite{more2011estimating}.
We first choose a differencing parameter (we plot $h=10^{-6}$) and number of points $m\in\{6, 12, 24\}$. 
In each trial of the experiment, we generate a random base point $y^0$ uniformly at random from the hypercube defined by $\{x\in\Reals^{10}: \|x\|_{\infty} \leq 10\}$ and a random unit-length direction $d\in\Reals^{10}$, chosen by normalizing a random zero-mean Gaussian vector. 
We then evaluate the noisy function $f(x)$ at each of $\{y^0 + jhd: j=0,1,\dots, m-1\}$ and run \texttt{ECNoise} on this set of function values to obtain a noise estimate; this corresponds to the ``standard" use of \texttt{ECNoise}, as intended in the original work \cite{more2011estimating}.
In each trial of the experiment, we then keep the same base point $y^0$ but generate a set of points
$\{y^0 + d_j: j=0,1,\dots,m-1\}$, where $d_j$ is drawn uniformly at random from the hypercube defined by $\{d\in\Reals^{10}: \|d\|_{\infty} \leq h\}$. 
This latter random selection represents an arbitrary (not collinear) set of points that are still somehow bounded by a radius $h$. 
This set of function values is then also provided to \texttt{ECNoise} to provide a separate noise estimate. 
We then compare the two distributions of relative noise estimates ($\xi(y^0)/f(y^0)$) obtained over $10^4$ many trials. 
In \Cref{fig:histogram} we show splined histograms of the results of running this experiment once with $h=10^{-6}$ and all three values of $m$. 
We emphasize how similar these distributions appear.
Additionally, a two-sample Kolmogorov--Smirnov test fails to reject the null hypothesis that distributions coming from the two distinct modes of point generation are equal. For the $m=6, 12, 24$ cases respectively, the corresponding $p$-values are $p = 0.9051, 0.8421, 0.2791$.

\begin{figure}
    \centering
    \includegraphics[width=.4\linewidth]{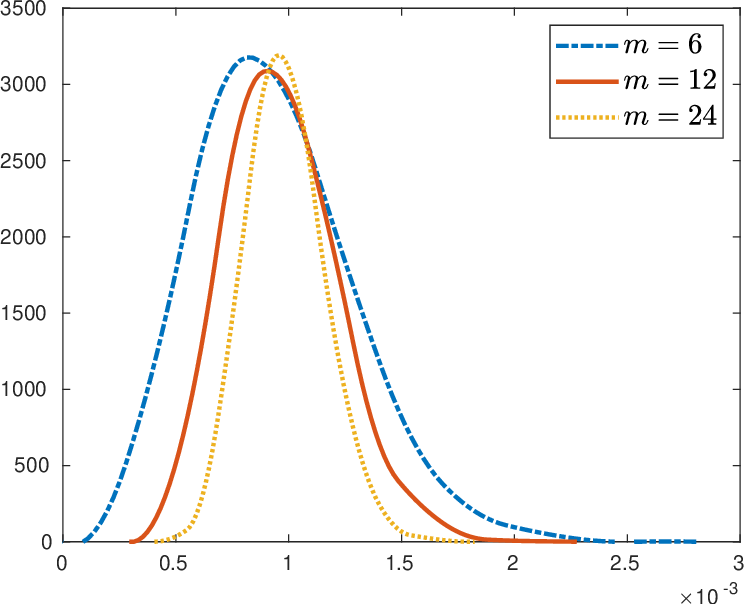}
    \includegraphics[width=.4\linewidth]{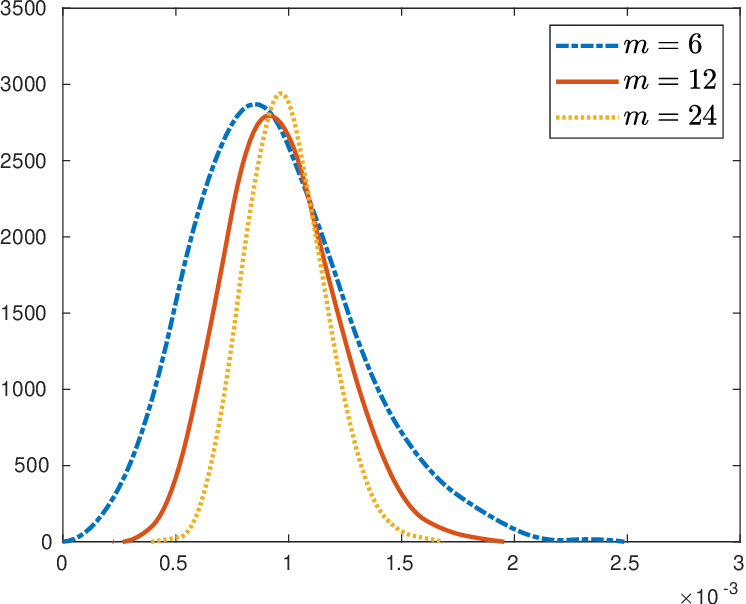}

    \caption{\label{fig:histogram} Splined histograms of relative noise values returned from running the described experiment. 
    Left figure corresponds to running \texttt{ECNoise} in ``standard" mode; right figure corresponds to running \texttt{ECNoise} with arbitrary function values.}
\end{figure}

For brevity in this short manuscript, we do not provide further evidence for this phenomenon. 
We do, however, comment that the seemingly unimportant nature of the geometry of points provided to \texttt{ECNoise} is not restricted to this simple problem, and we have observed this on a variety of noisy problems. 
In the next section we provide a framework that describes arbitrary point sets as lying on a parametric curve defined by a single time variable, and we prove results concerning the output of \texttt{ECNoise} under this interpretation. 
We demonstrate that these results recover the results of \texttt{ECNoise} proven in \cite{more2011estimating} as a special case.

\section{\texttt{ECNoise} on Parametric Curves}\label{sec:theory}
Suppose we are given a set of points $Y=\{y^0, y^1, \dots, y^m\}\subset\Reals^n$. 
We will construct a specific parametric curve $p(t): [0,m]\to\Reals^n$ for some $m\in\Nats$.
Our construction will ensure that $p(t)$ is smooth, that is, infinitely times continuously differentiable. 
We want $p(t)$ to interpolate at each point in $Y$; that is, we want
\begin{equation}
    \label{eq:p_interpolates}
    p(j) = y^j \quad \forall j=0,1,\dots,m.
\end{equation}
We will provide the rest of the construction later, after we have established a few key theorems that depend only on $p$ being smooth and satisfying \eqref{eq:p_interpolates}. 
Extending the notation in \eqref{eq:kth_order_diff}, we denote a $k$th-order difference of $f(p(t))$ recursively with a base case
$\Delta^0 f(p(t)) = f(p(t))$ and the recursively defined
$$\Delta^{k+1} f(p(t)) = \Delta[\Delta^k f(t)] = \Delta^k f(p(t + 1)) - \Delta^k f(p(t)).$$

The proof of the following result is essentially unchanged from \cite{more2011estimating}[Theorem 2.2].
\begin{theorem}\label{thm:variance}
    If $\{\xi(p(j)) : j = 0,1,\dots,m\}$ are iid realizations of the random variable $\Xi(y^0,\delta)$,
    then
    $$\gamma_k \Var\left[\Delta^k\xi(p(0))\right] = \epsilon_f^2 \quad \forall k=1,2,\dots,m,$$
    where
    $\gamma_k = (k!)^2/(2k)!$ for each $k=1,2,\dots,m$.
\end{theorem}

\begin{proof}
    We will first show by induction that
    \begin{equation}
    \label{eq:induction}
    \Delta^k\xi(p(t)) = \displaystyle\sum_{j=0}^k (-1)^j\binom{k}{j} \xi(p(t + k-j)) \quad \forall k=0,1,\dots,m.
    \end{equation}
    Indeed, the base cases of $k=0$ and $k=1$ are trivial. 
    To show that \eqref{eq:induction} holds in general, suppose \eqref{eq:induction} holds for $k-1 < m$, that is,
    $$\Delta^{k-1}\xi(p(t)) = \displaystyle\sum_{j=0}^{k-1} (-1)^j\binom{k-1}{j} \xi(p(t + k-1-j)). $$
    Then, by the definition of $\Delta^k$,
    $$
    \begin{array}{rl}
    \Delta^k \xi(p(t)) = & \Delta^{k-1}\xi(p(t+1)) - \Delta^{k-1}\xi(p(t))\\
    = & \displaystyle\sum_{j=0}^{k-1} (-1)^j\binom{k-1}{j} 
    \xi(p(t + k-j)) - 
    \displaystyle\sum_{j=0}^{k-1} (-1)^j\binom{k-1}{j}
    \xi(p(t+k-1-j))\\
    =& 
    \displaystyle\sum_{j=0}^{k-1} (-1)^j\binom{k-1}{j} 
    \xi(p(t + k-j))
    -
    \displaystyle\sum_{j=1}^{k} (-1)^{j-1}\binom{k-1}{j-1}
    \xi(p(t+k-j))
    \\
    =&
    \displaystyle\sum_{j=0}^{k-1} (-1)^j\binom{k-1}{j} 
    \xi(p(t + k-j))
    +
    \displaystyle\sum_{j=1}^{k} (-1)^{j}\binom{k-1}{j-1}
    \xi(p(t+k-j))\\
    =&
    (-1)^0\binom{0}{0}\xi(p(t+k))
    + \displaystyle\sum_{j=1}^{k-1} 
    \left(\binom{k-1}{j} + \binom{k-1}{j-1}\right)  \xi(p(t+k-j))
    + (-1)^k\binom{k-1}{k-1}\xi(p(t))\\
    = & 
    (-1)^0\binom{k}{0}\xi(p(t+k))
    + \displaystyle\sum_{j=1}^{k-1} 
    \binom{k}{j}   \xi(p(t+k-j))
    + (-1)^k\binom{k}{k}\xi(p(t)),
    \\
    \end{array}
    $$
    where the last equality employed Pascal's identity. This proves the induction hypothesis \eqref{eq:induction}. 

    By the iid assumption and well-known properties of variance, it follows from \eqref{eq:induction} that
    $$\Var[\Delta^k \xi(0)] = \displaystyle\sum_{j=0}^k \binom{k}{j}^2 \Var[\xi(p(k-j))] 
    = \epsilon_f^2 \displaystyle\sum_{j=0}^k \binom{k}{j}^2. 
    $$
    The desired result now follows from a standard combinatorial identity,
    $$\displaystyle\sum_{j=0}^k\binom{k}{j}^2 = \displaystyle\frac{(2k)!}{(k!)^2}.$$
\end{proof}

\begin{theorem}
\label{thm:noise_estimate}
    Suppose $\{\xi(p(j)): j=0,1,\dots,m\}$ are iid realizations of the random variable $\Xi(y^0,\delta)$.
    If $f_s$ is continuous at $p(0)$, then 
    $$\epsilon_f^2 = \gamma_k \displaystyle\lim_{h\to 0^+} \Exp[\Delta^k f(p(0))^2]  \quad k=1,2,\dots,m,$$
    where $h = \max\{\|y^1 - y^0\|,\|y^2 - y^1\|, \dots, \|y^m - y^{m-1}\|\}$. 
\end{theorem}

\begin{proof}
    By the iid assumption and \eqref{eq:induction}, 
    \begin{equation}\label{eq:zeromean}
    \Exp[\Delta^k \xi(p(0))] =
    \Exp_{\xi\sim\Xi(y^0,\delta)}[\xi]\displaystyle\sum_{j=0}^k (-1)^j\binom{k}{j} = 0.
    \end{equation}
    The last equality in \eqref{eq:zeromean} follows by noting that the summation in \eqref{eq:zeromean} can be viewed as the $k$th-order Taylor expansion of the function $(1+s)^k$ about $s=-1$. 
    In other words, we have shown that the $k$th-order difference at $f(p(0))$ has zero mean. 
    With \eqref{eq:zeromean} in hand, 
    \begin{equation}\label{eq:int2}
    \begin{array}{rl}
    \Exp[\Delta^k f(p(0))^2] = &
    \Delta^k f_s(p(0))^2 + 2\Delta^k f_s(p(0))\Exp[\Delta^k\xi(p(0))] +
    \Exp[\Delta^k\xi(p(0))^2]\\
    =& \Delta^k f_s(p(0))^2 + \Var[\Delta^k\xi(p(0))]\\
    = & \Delta^k f_s(p(0))^2 + \displaystyle\frac{\epsilon_f^2}{\gamma_k},
    \end{array}
    \end{equation}
    where the last equality used the result of \Cref{thm:variance}.
    The result follows by multiplying both sides of \eqref{eq:int2} by $\gamma_k$
    and then noting that if 
    $h\to 0^+$, then necessarily $\|y^{j+1} - y^j\| \to 0^+$ for all $j=0,1,\dots,m-1$,
    and so by the definition of $k$th differences and the assumed continuity of $f_s \circ p$, 
    $\Delta^k f_s(p(0)) \to 0$. 
\end{proof}

\Cref{thm:noise_estimate} provides a guarantee that, given $m+1$ points and provided $f_s$ (and our parametric curve, $p$, interpolating those $m+1$ points) is continuous, then the noise level $\epsilon_f$ is expressible as a known constant multiple of the square root of the analytical expectation of the square of any $k$th difference of $f\circ p$ for $k\leq m$. 
By itself, this theorem says nothing about the quality of this approximation (for instance, the variance of this estimator). 
As in \cite{more2011estimating}, we can demonstrate proportionally stronger results concerning approximation quality as a function of the assumed smoothness of $f_s$. 
By our smoothness assumption on $p(t)$, we have
 by application of the chain rule that
\begin{equation}\label{eq:first_derivative}
\begin{array}{rl}
(f_s \circ p)^{(1)}(t) = & \displaystyle\sum_{i=1}^n \partial_i f_s(p(t))[p^{(1)}(t)]_i,
\end{array}
\end{equation}
where $\partial_i$ denotes the $i$th partial derivative and where we use parenthetical superscripts for ``time" derivatives (those derivatives with respect to $t$)
and where we let $[\cdot]_i$ denote the $i$th entry of a vector. 

We pause to consider the special case of equally spaced points $\{y^j = y^0 + hjd, j=0,1,\dots,m\}$, given a vector $d$ satisfying $\|d\|=1$ and some $h>0$; this is precisely the case considered in \cite{more2011estimating}.
In this case, the most natural choice of $p(t)$ satisfying the interpolation condition \eqref{eq:p_interpolates} is given by
\begin{equation}
    \label{eq:equallyspaced}
    p(t) = y^0 + thd.
\end{equation}
Such a $p$ is obviously smooth; and in view of \eqref{eq:first_derivative}, we see that 
$f_s^{(1)}(p(t)) = h \langle \nabla f_s(p(t)), d\rangle$, that is,
the first derivative of $f_s\circ p$ is the directional derivative of $f_s$ in the direction $d$ scaled by $h$.
Moreover, in this special case of equally spaced points,
we have that the natural choice of $p(t)$ in \eqref{eq:equallyspaced} satisfies $p^{(2)}(t)=0$ everywhere. 
Thus, in the special case, we obtain from repeated applications of the chain rule that 
\begin{equation}
    \label{eq:kth_equallyspaced}
    (f_s \circ p)^{(k)}(t) = h^k \displaystyle\sum_{i_1=1}^n 
    \displaystyle\sum_{i_2=1}^n
    \dots
    \displaystyle\sum_{i_k=1}^n
    \partial_{i_1}\partial_{i_2}\dots\partial_{i_k} f_s(p(t)) 
    [d]_{i_1}[d]_{i_2}\dots [d]_{i_k}.
\end{equation}
This observation will momentarily lead to the same result reached in the second part of \cite{more2011estimating}[Theorem 2.3], restated in \Cref{thm:kcd_equallyspaced}. We first state a lemma. 

\begin{lemma}
\label{lem:mvt}
    Suppose a function $g:\Reals\to\Reals$ is $k$-times differentiable on $[0,m]$. Then there exists $\zeta\in (0,m)$ satisfying
    \begin{equation}\label{eq:dividif}
\Delta^k g(0) = g^{(k)}(\zeta).
\end{equation}
\end{lemma}

The proof of \Cref{lem:mvt} is not trivial, 
but the result is fairly well known, and proving it here would be distracting. 
A proof involving the Newton form of an interpolating polynomial may be found surrounding \cite{quarteroni2010numerical}[Equation 8.21].\footnote{We comment additionally for the reader that checks this citation that because our parameteric curve $p$ interpolates $\{y^0, y^1, \dots, y^m\}$ at equally spaced points $t = 0,1,\dots m$, it is easily proved by induction that the $k$th \emph{divided difference}, denoted $g[0,1,\dots,k]$ in \cite{quarteroni2010numerical}, satisfies the relation
$g[0,1,\dots,k] = \frac{\Delta^k g(0)}{k!}.$}

\begin{theorem}
\label{thm:kcd_equallyspaced}
    Suppose $\{\xi(p(j)): j = 0,1,\dots,m\}$ are iid realizations of the random variable $\Xi(y^0,\delta)$. 
    Additionally, suppose $y^j = y^0 + hjd$ for each of $j=0,1,\dots,m$,
    and define $p(t)$ via \eqref{eq:equallyspaced}. 
    If $f_s$ is $k$-times continuously differentiable on $[0,m]$, then 
    \begin{equation}\label{eq:rate}
    \displaystyle\lim_{h\to 0^+} \frac{\gamma_k\mathbb{E}[\Delta^k(f(y^0))^2] - \epsilon_f^2}{h^{2k}} = \gamma_k f_s^{(k)}(0; d)^2,
    \end{equation}
    where we have made the notational choice
    $$f_s^{(k)}(\zeta; d) = 
    \displaystyle\sum_{i_1=1}^n 
    \displaystyle\sum_{i_2=1}^n
    \dots
    \displaystyle\sum_{i_k=1}^n
    \partial_{i_1}\partial_{i_2}\dots\partial_{i_k} f_s(p(\zeta)) 
    [d]_{i_1}[d]_{i_2}\dots [d]_{i_k}.$$
\end{theorem}

\begin{proof}
Because $f_s \circ p$ is $k$-times differentiable on $[0,m]$, we get from \Cref{lem:mvt} that there exists
$\zeta\in(0,m)$ so that \eqref{eq:dividif} holds.
Thus, continuing from \eqref{eq:int2} and sequentially employing \eqref{eq:dividif}  and \eqref{eq:kth_equallyspaced}, 
 \begin{equation}\label{eq:int3}
 \begin{array}{rl}
    \Exp[\Delta^k f(p(0))^2] = &  
    \Delta^k f_s(p(0))^2 + \displaystyle\frac{\epsilon_f^2}{\gamma_k}\\
    = & (f_s \circ p)^{(k)}(\zeta)^2  + \displaystyle\frac{\epsilon_f^2}{\gamma_k}\\
    = & h^{2k} \left[
    f_s^{(k)}(\zeta; d)
    \right]^2 + \displaystyle\frac{\epsilon_f^2}{\gamma_k}.
\end{array}
\end{equation}
The result follows after algebraic manipulation and noting that $\zeta\to 0$ as $h\to 0^+$. 
\end{proof}

\Cref{thm:kcd_equallyspaced} effectively recovers the second part of the result of \cite{more2011estimating}[Theorem 2.3]
because if we suppose $f_s$ has a one-dimensional domain and we choose $d=1\in\Reals^1$, then $f_s^{(k)}(0; d)$ would be simply $f_s^{(k)}(y^0)$. 

In the general case of an arbitrary set of points $Y$, we will not be able to construct a $k$-times differentiable curve $p$ such that the second derivative $p^{(2)}(t) = 0$ for all $t\in[0,m]$. Moreover, it will not hold that any higher-order derivative satisfies $p^{(k)}(t)=0$ for all $t\in[0,m]$. 
It is precisely these zero higher-order derivatives that enabled the simple expression for the $k$th derivative of the composition $f_s\circ p$ in \eqref{eq:kth_equallyspaced}. 
For general parameterized curves $p(t)$, \eqref{eq:first_derivative} always holds true; but after subsequent applications of yhe product and chain rule, higher-order derivative expressions become increasingly more unwieldy.
For instance, the second derivative of $(f_s\circ p)$ with a general $p(t)$ can be computed as 
$$(f_s\circ p)^{(2)}(t) = \displaystyle\sum_{i=1}^n\left[\left[\displaystyle\sum_{j=1}^n \partial_j \partial_i f_s(p(t))[p^{(1)}(t)]_i[p^{(1)}(t)]_j \right] + \partial_i f_s(p(t)) [p^{(2)}(t)]_i \right].$$
Higher-order derivatives will involve proportionally more nested sums, but we can minimally see that an expression for $(f_s\circ p)^{(k)}(t)$ will involve the summands 
\begin{equation}\label{eq:s1}
S_1 := \displaystyle\sum_{(i_1,i_2,\dots, i_k)\in\{1,2,\dots,n\}^k} \partial_{i_1}\partial_{i_2}\cdots\partial_{i_k} f_s(p(t))[p^{(1)}(t)]_{i_1} [p^{(1)}(t)]_{i_2} \dots [p^{(1)}(t)]_{i_k}
\end{equation}
and
\begin{equation}\label{eq:sk}
S_k := \displaystyle\sum_{i=1}^n \partial_i f_s(p(t))[p^{(k)}(t)]_i.
\end{equation}
We now fix a choice of $[p(t)]_i$ as the \emph{unique} degree-$m$ polynomial that satisfies $[p(j)]_i = y^j_i$ for $j=0,\dots,m$.
The Lagrange form of such a polynomial is given by 
$$p_i(t) = \displaystyle\sum_{j=0}^m y^j_i L_j(t), 
\quad \text{ where } \quad 
L_j(t) := \displaystyle\prod_{\ell=0, \ell\neq j}^m \frac{t - \ell}{j - \ell}. $$
Thus,
if we let
$$h = \displaystyle\max_{i=1\dots,n} \max_{j=0,1,\dots,m} y^j_i,$$
we see that for all $k\in 0, 1, \dots, m$,
$$\displaystyle\max_{j\in 0,1,\dots, m} \max_{t\in[0,m]}
|L_j^{(k)}(t)|$$
is independent of $h$,
and so, for any $k\in 0, 1,\dots,m$ and for any $t\in[0,m]$,
$|p_i^{(k)}(t)| \leq Ch$
for some $C > 0$.
Thus, as $h\to 0^+$, $S_1\in\mathcal{O}(h^k)$ in \eqref{eq:s1} and $S_k\in\mathcal{O}(h)$ in \eqref{eq:sk}. 
The following theorem is immediate. 

\begin{theorem}
    \label{thm:kcd}
    Suppose $\{\xi(p(j)): j = 0,1,\dots,m\}$ are iid realizations of the random variable $\Xi(y^0,\delta)$. 
    Let $p(t)$ be a $k$-times continuously differentiable parameterization that satisfies
    $[p(j)]_i = y^j_i$ for $i=1,2,\dots,n$ and $j=0,1,\dots,m$. 
    Without loss of generality, let $y^0=0$. 
    If $f_s$ is $k$-times continuously differentiable on $[0,m]$, then there exist constants $\{C_\ell \geq 0: \ell = 0, 1, \dots k\}$ so that
    \begin{equation}
        \label{eq:rateless}
        \displaystyle\lim_{h\to 0^+} 
       \displaystyle\frac{\left|\gamma_k\mathbb{E}[\Delta^k(f(y^0))^2] - \epsilon_f^2\right|}
        {\left[\displaystyle\sum_{\ell=1}^k 
C_\ell h^\ell \right]^2}
        \leq
    \gamma_k f_s^{(k)}(y^0)^2,
    \end{equation}
     where $h = \max\{\|y^1\|_{\infty},\|y^2\|_{\infty}, \dots, \|y^m\|_{\infty}\}$. 
\end{theorem}
\begin{proof}
    The proof of \Cref{thm:kcd_equallyspaced} still holds up until the second line of \eqref{eq:int3}.
    Then, the remainder of the proof follows by our observations concerning the $k$th derivative $(f_s\circ p)^{(k)}$ and its dependence on $h$. 
\end{proof}

\section{Understanding \Cref{thm:kcd}} \label{sec:understanding}
Unfortunately, on the surface, \Cref{thm:kcd} implies that
as $h\to 0^+$, the denominator of \eqref{eq:rateless} is effectively $\mathcal{O}(h^2)$;
that is, the convergence rate of the estimator is independent of $k$, a disappointing result in light of \Cref{thm:kcd_equallyspaced}. 
We stress that \Cref{thm:kcd} is indeed pessimistic, since we have essentially assumed that the derivatives of $f_s$ and the (derivatives of) Lagrange polynomials evaluated at a sequence of $\zeta$ that would correspond to a convergent sequence $h\to 0^+$ are always large enough to attain some bound. 
One would more likely expect an ``average case" behavior to occur, where average is with respect to some unspecified and complicated distribution on both $k$-times continuously differentiable functions $f_s$, and on a realization of points $\{y^0, y^1, \dots, y^m\}$ corresponding to each $h$ in the sequence $h\to 0^+$. 

Still, \Cref{thm:kcd} highlights that, as one might intuitively expect, the price one pays for estimating noise along a parametric curve, as opposed to the ``perfect" equally spaced, collinear geometry exemplified in \Cref{thm:kcd_equallyspaced}, is that the quality of the estimates returned by \texttt{ECNoise} is more strongly dependent on the distance parameter $h$.
Moreover, via the constants $C_{\ell}$, we expect noise level estimates to be more dependent on the problem dimension $n$, \emph{provided all of the mixed partial derivatives of order up to $k\leq m$ are nontrivial}. 

To illustrate this dependence on $h$ and higher-order derivatives with a synthetic problem, we consider, for $m=6$, the function
\begin{equation}
    f_s(x) = \left(\displaystyle\sum_{i=1}^n x_i\right)^m
    \label{eq:synthetic}
\end{equation}
and generate additive noise $\xi\sim\mathcal{N}(0,10^{-3})$
so that $f(x) = f_s(x) + \xi$. 
We then run the same experiment as in \Cref{sec:motivation} but with this $f_s(x)$ and varying the values of $n$ and $h$.  
Results are shown in \Cref{fig:bigfigure}.
This simple experiment exhibits the trends we expected to observe: as $n$ and $h$ increase, the gap in performance between the standard use of \texttt{ECNoise} and the use of arbitrary points as input to \texttt{ECNoise} increases. 
Remarkably, and as is expected from \Cref{thm:kcd}, if we repeat the same experiment but with a quadratic objective $f_s(x) = x^\top x$, the empirical cdfs for the two uses of \texttt{ECNoise} virtually overlap for all values of $h$ and $n$.
Moreover, a two-sample Kolmogorov-Smirnov test fails to reject the null hypothesis that the underlying cdfs are equal for all values of $h$ and $n$. 
In light of \Cref{thm:kcd}, this outcome may be explained by observing that all derivatives of $f_s(x) = x^\top x$ of order greater than 2 are zero.

\begin{figure}[ht!]
    \centering
    \includegraphics[width=.99\textwidth]{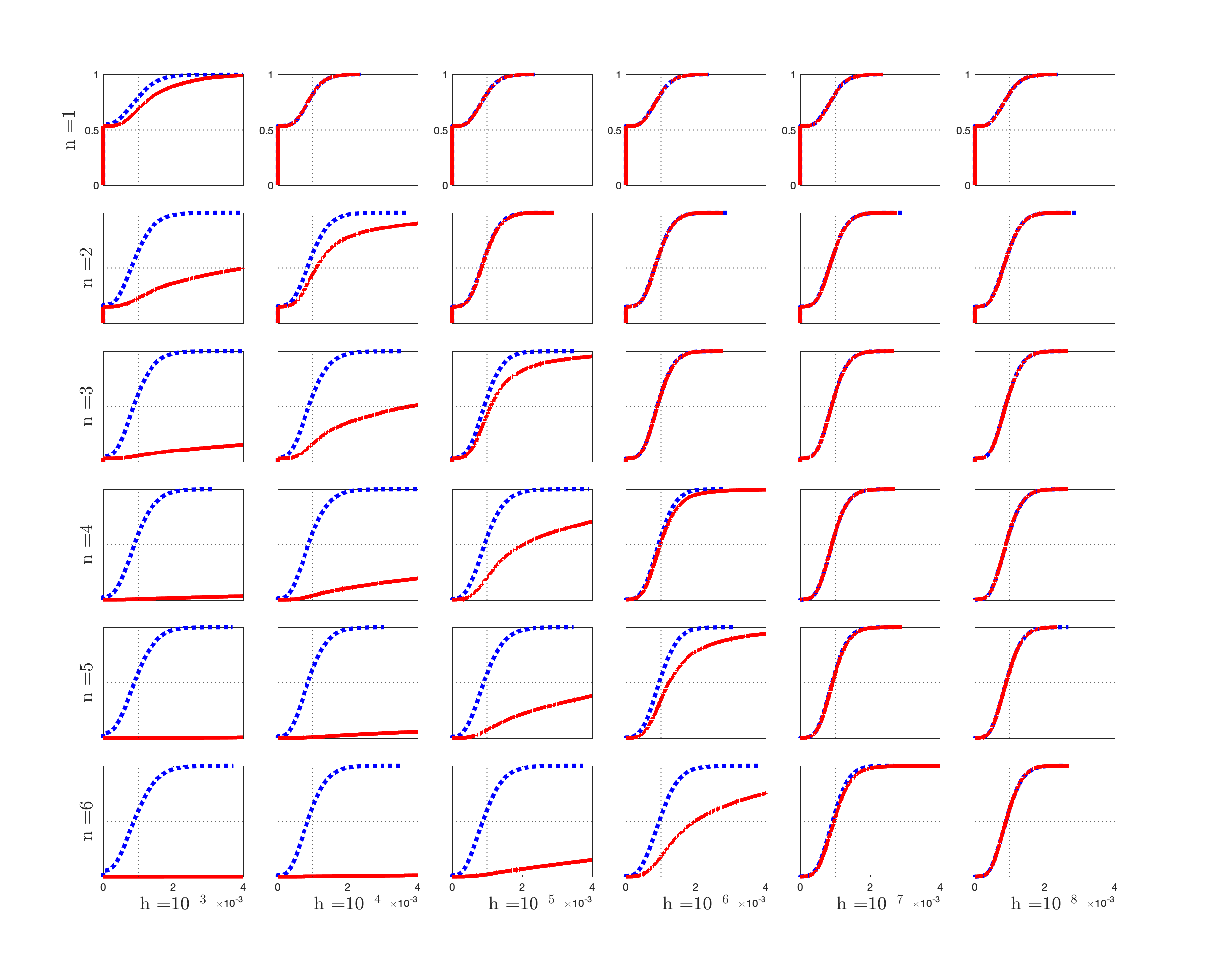}
    \caption{ Comparing the ``standard" use of \texttt{ECNoise} (in dashed blue) to arbitrary point selection (in solid red) on the synthetic problem with smooth part defined by \eqref{eq:synthetic} over varying values of $h$ and $n$.
    These plots are empirical cumulative density functions, over 10,000 trials, of the noise level estimate returned by \texttt{ECNoise}. 
    We note that \texttt{ECnoise} may decline to provide a noise level estimate if it detects $h$ is too large. 
    In this case the noise level estimate is set to 0. 
    This accounts for, in the larger values of $h$ on the left of this grid, the left-hand side of the empirical cumulative distribution function beginning at a probability value greater than 0. 
    The limit of the $x$-axis is chosen as 4 times the correct value of $\epsilon_f = 10^{-3}$, which was the threshold specified for successful noise estimation specified in \cite{more2011estimating}. 
    We also show, in each subplot via black dotted lines, the location of $\epsilon_f=10^{-3}$ on the $x$-axis and the location of $0.5$ on the $y$-axis, since we expect the median of the noise level estimate distribution to lie around $10^{-3}$.}
    \label{fig:bigfigure}
\end{figure}

\section{A Practical Method for Down-selecting Sample Points}\label{sec:milp}
\Cref{thm:kcd} additionally provides \emph{practical} insight into how to use \texttt{ECNoise} with arbitrary input points.
To simulate how we anticipate using \texttt{ECnoise} within a typical derivative-free optimization algorithm for noisy objective functions, we suppose that we have previously evaluated an objective function at points $\{y^0, y^1, \dots, y^M\}$ and we wish to choose a subset of size $m$ of them to provide as input to \texttt{ECNoise}, in addition to $y^0$. 
Whereas we previously considered the Lagrange form of each interpolating polynomial $p_i(t)$, we now consider the Newton form of each $p_i$, that is,
$$p_i(t) = \displaystyle\sum_{j=0}^m [y_i^0, \dots, y_i^j] \omega_j(t), 
\quad
\text{ where }
\quad
\omega_j(t) = \displaystyle\prod_{\ell=0}^j (t - \ell)
$$
and $[y_i^0, \dots, y_i^j]$ denotes the $j$th divided difference of the data $\{y_i^0, y_i^1, \dots, y_i^m\}$ (see a previous footnote). 
In \Cref{sec:understanding} we considered how the $C_\ell$ terms appearing in \eqref{eq:rateless} were functions of higher-order derivatives of $f_s$. 
We see from this analysis, however, that each $C_\ell$ also depends on derivatives of $p_i(t)$. 
In light of the Newton form for $p_i(t)$, to bound all derivatives of $p_i(t)$, it is beneficial to choose $\{y^0, y^1,\dots, y^m\}$ in such a way that it minimizes the magnitudes of all $k$th divided differences. 

To realize this minimization of coefficients, we provide
a mixed-integer linear program, which may be viewed as an assignment problem with additional constraints. 
In particular, given the $M$ sample points, we wish to assign at most one point to each of $j=1,\dots,m$ in the time variable for the parametric curve. 
We will denote by $\{z_{j,q}: j\in\{1,\dots,m\}, q\in\{1,\dots,M\}\}$ the set of binary variables which take the value 1 when point $q$ is assigned to time $j$ in the curve, and 0 otherwise.
For a fixed selection of $m$ points denoted $c^1,\dots,c^m$, 
we can describe, for the $i$th coordinate, the $k$th divided difference $[c_i^0, c_i^1, \dots, c_i^k]$ for each of $k=0,1,\dots,m$ 
via the formula
$$[c_i^0,c_i^1,\dots,c_i^k] = \displaystyle\sum_{j=0}^k (-1)^j {k \choose j} / j! c_i^{j},$$
which is a linear equality in $c_i^{0},c_i^{ 1},\dots,c_i^{m}$. 
By our reasoning at the end of \Cref{sec:theory}, a reasonable choice of an objective for the model is
$$\displaystyle\min_{z} \max_{i\in\{1,\dots,n\}} \max_{j\in\{1,\dots,m\}} \left|[c_i^{ 0}, c_i^{1}, \dots, c_i^{ j}]\right|.$$
The benefit of employing this particular objective is that the optimization problem remains a mixed-integer linear program (MILP).
While MILPs are well known to be NP-hard, in standard use cases of expensive derivative-free optimization, we do not expect either the number of available function evaluations $M > m$ or the problem dimension $n$ to be particularly large, so this modified assignment problem will generally be tractable, especially in settings where obtaining additional function evaluations is very expensive. 
In fact, in the experiment in this section, we solved \eqref{eq:milp_R} in \texttt{BARON} \cite{baron};
on a personal computer, no solution took significantly more than one second.  
For the sake of  clarity, we provide the full statement of the MILP in \eqref{eq:milp}, noting that the absolute values and maxima present in the objective function of \eqref{eq:milp} can be formulated as linear inequalities: 
\begin{equation}
\label{eq:milp}
\begin{array}{rl}
\displaystyle\min_{z} & \displaystyle\max_{i\in\{1,\dots,n\}} \displaystyle\max_{j\in\{1,\dots,m\}} |\omega_{i,j}|\\
\text{subject to} & \displaystyle\sum_{j=1}^m z_{j,q} \leq 1 \quad \forall q = 1,\dots,M\\
& \displaystyle\sum_{q=1}^M z_{j,q} = 1 \quad \forall j=1,\dots,m\\
& c_i^j = \displaystyle\sum_{q=1}^M y^q_i z_{j,q} \quad \forall i = 1,\dots,n, \forall j=1,\dots,m\\
& \omega_{i,j} = \displaystyle\sum_{k=0}^j \left((-1)^k {j \choose k}/j!\right) c_i^j \quad \forall i = 1,\dots,n, \forall j=1,\dots,m .\\
\end{array}
\end{equation}

We additionally consider augmenting \eqref{eq:milp} by allowing a user to specify a budget of $R\leq m$ points that must be \emph{reused} from $\{y^1,y^2,\dots,y^M\}$, while each of the remaining $m-R$ \emph{free} points can be chosen to be any point $\phi\in\Reals^n$ satisfying $\|c^0 - \phi\|_{\infty} \leq h$. That is, we solve
\begin{equation}
    \label{eq:milp_R}
    \begin{array}{rl}
\displaystyle\min_{z, \phi} & \displaystyle\max_{i\in\{1,\dots,n\}} \displaystyle\max_{j\in\{1,\dots,m\}} |\omega_{i,j}|\\
\text{subject to} & \displaystyle\sum_{j=1}^m z_{j,q} \leq 1 \quad \forall q = 1,\dots,M\\
& \displaystyle\sum_{q=1}^M z_{j,q} \leq 1 \quad \forall j=1,\dots,m\\
& \displaystyle\sum_{j=1}^m \sum_{q=1}^M z_{j,q} = R\\
& -h\left(1 - \displaystyle\sum_{q=1}^M z_{j,q}\right) 
\leq \phi_i^j \leq 
h \left(1 - \displaystyle\sum_{q=1}^M z_{j,q}\right) \quad
\forall i = 1,\dots,n, \forall j=1,\dots,m\\
& c_i^j = \displaystyle\sum_{q=1}^M y^q_i z_{j,q} + \phi_i^j \quad \forall i = 1,\dots,n, \forall j=1,\dots,m\\
& \omega_{i,j} = \displaystyle\sum_{k=0}^j \left((-1)^k {j \choose k}/j!\right) c_i^j \quad \forall i = 1,\dots,n, \forall j=1,\dots,m,\\
\end{array}
\end{equation}
which is equivalent to \eqref{eq:milp} when $R=m$. 

In \Cref{fig:compare_reuse} we rerun the experiment that corresponds
with the $n=6, h=10^{-6}$ case in \Cref{fig:bigfigure}.
This time, in each trial, after generating $y^0$, we choose $M=50$ points from a uniform random distribution on $\{x: \|x - y^0\|_{\infty} \leq h$ and solve \eqref{eq:milp_R} to choose a subset of size $m$ with varying reuse budget $R$.
We additionally show the results of the same experiment but with $m=12$ points used in \texttt{ECNoise}, instead of $m=6$.
As expected, as $R$ decreases---that is, we choose to evaluate increasingly more points outside of those given in $\{y^1,\dots,y^M\}$---the empirical cdf approaches that corresponding to the standard use of \texttt{ECNoise}. 
The improvement from allowing a single free point is particularly stark in the $m=12$ case. 

\begin{figure}
    \centering
    \includegraphics[width=.49\textwidth]{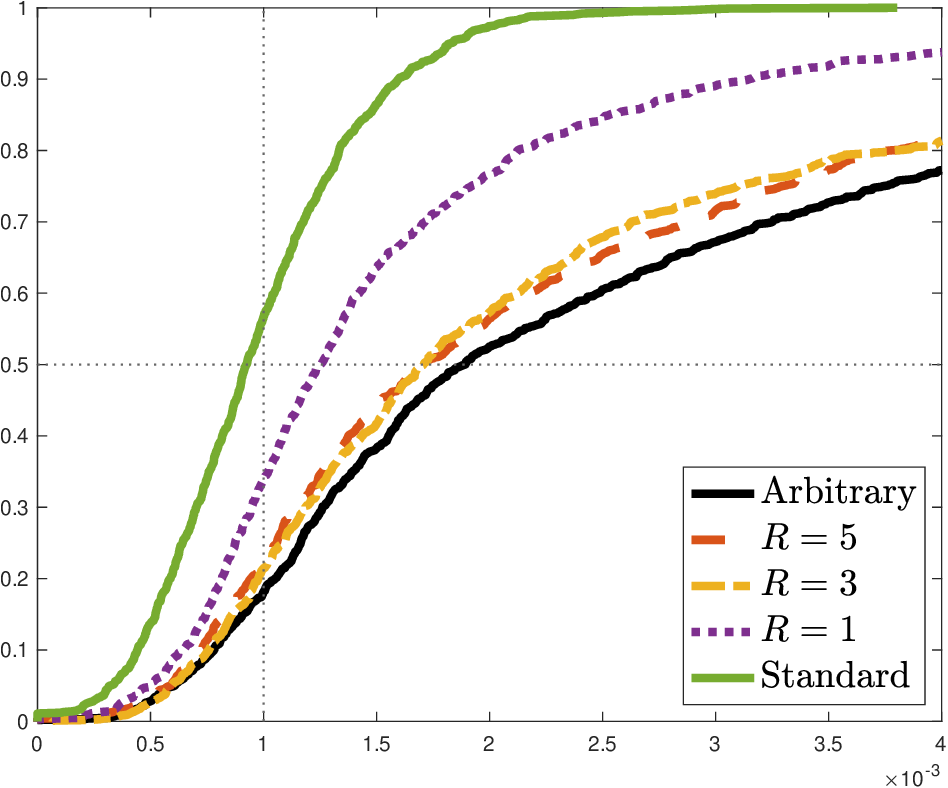}
    \includegraphics[width=.49\textwidth]{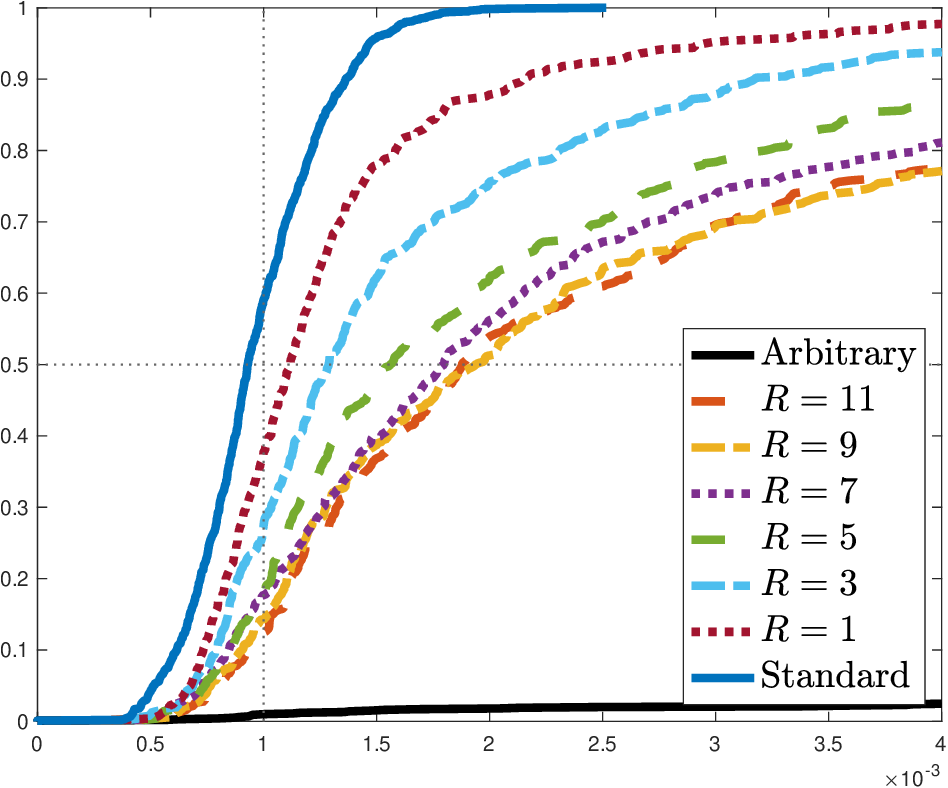}
    \caption{Repeating the experiment in \Cref{fig:compare_reuse} corresponding to $n=6$ and $h=10^{-6}$, with the same empirical cdfs.
    For reference, we continue to display empirical cdfs for the ``standard" use of \texttt{ENnoise} and the use of \texttt{ECNoise} with a randomly generated set of points as in the prior experiment.
    This time, in each trial, we also solve the MILP in \eqref{eq:milp_R} 
    with the specified $R$ and the identical arbitrary point set used to generate the noise estimate in the arbitrary case. 
    The left plot shows $m=6$, while the right plot shows $m=12$. 
    \label{fig:compare_reuse}}
\end{figure}

\section*{Acknowledgments} This work was supported in part by the U.S.~Department of Energy, Office of Science, Office of Advanced Scientific Computing Research Applied Mathematics under Contract No.~DE-AC02-06CH11357 .

\bibliographystyle{unsrt}
\bibliography{main}

\framebox{\parbox{.90\linewidth}{\scriptsize The submitted manuscript has been created by
        UChicago Argonne, LLC, Operator of Argonne National Laboratory (``Argonne'').
        Argonne, a U.S.\ Department of Energy Office of Science laboratory, is operated
        under Contract No.\ DE-AC02-06CH11357.  The U.S.\ Government retains for itself,
        and others acting on its behalf, a paid-up nonexclusive, irrevocable worldwide
        license in said article to reproduce, prepare derivative works, distribute
        copies to the public, and perform publicly and display publicly, by or on
        behalf of the Government.  The Department of Energy will provide public access
        to these results of federally sponsored research in accordance with the DOE
        Public Access Plan \url{http://energy.gov/downloads/doe-public-access-plan}.}}

\end{document}